\documentclass{amsart}
\usepackage{fullpage}

\def \ll{\mathfrak{s}}
\def \lie{\mathfrak{l}}
\def \subalg{\mathfrak{h}}
\def \m{\mathfrak{m}}

\usepackage{amsfonts}
\usepackage{amsmath, amscd, amssymb}

\newcommand{\Mlt}{\mathrm{Mlt} }
\newcommand{\PMlt}{\mathrm{PMlt} }
\newcommand{\Sab}{\mathrm{Sab} }
\newcommand{\spann}{\mathrm{span} }

\newcommand{\Q}{\mathbb{Q}}
\newcommand{\A}{\mathcal{A}}
\newcommand{\B}{\mathcal{B}}
\newcommand{\bigL}{\mathcal{L}}
\newcommand{\F}{\mathcal{F}}
\newcommand{\kk}{\boldsymbol{k}}
\newcommand{\Prim}{\mathrm{Prim}\, }

\newtheorem{theorem}{Theorem}
\newtheorem{proposition}[theorem]{Proposition}
\newtheorem{lemma}[theorem]{Lemma}
\newtheorem{corollary}[theorem]{Corollary}

\theoremstyle{remark}
\newtheorem*{remark}{Remark}

\font\cyrillic=wncyi10
\newcommand{\SU}{\mathop{\hbox{{\cyrillic UX}}}}

\title{Nilpotent Sabinin algebras}

\author{J.~Mostovoy, J.M.~Perez-Izquierdo, I.P.~Shestakov}
\begin{document}

\maketitle

\begin{abstract} In this note we establish several basic properties of nilpotent Sabinin algebras. Namely, we show that nilpotent Sabinin algebras (1) can be integrated to produce nilpotent loops, (2) satisfy an analogue of the Ado theorem, (3) have nilpotent Lie envelopes. We also give a new set of axioms for Sabinin algebras. These axioms reflect the fact that a complementary subspace to  a Lie subalgebra in a Lie algebra is a Sabinin algebra. Finally, we note that the non-associative version of the Jennings theorem produces a version of Ado  theorem for loops whose commutator-associator filtration is of finite length.
\end{abstract}

\section{Introduction} 

Lie theory for non-associative products has the same broad outlines as the usual Lie theory. Roughly speaking, the triangle consisting of Lie groups, Lie algebras and Hopf algebras is replaced in the non-associative setting with the triangle consisting of local loops, Sabinin algebras and non-associative Hopf algebras, see \cite{MP3} and \cite{MPSh} for the details.

There are several features of the infinitesimal theory of Lie groups that do not extend to the general non-associative case. While  finite-dimensional Lie algebras can always be integrated to Lie groups, for general Sabinin algebras only a local integration procedure is available. There is no reason to expect that this local procedure can be globalized since many natural examples of non-associative products are of local nature. Another such feature is the existence of finite-dimensional enveloping associative algebras for finite-dimensional Lie algebras (the Ado theorem). There are finite-dimensional Sabinin algebras (simple Lie triple systems, for instance) which do not admit finite-dimensional envelopes, see \cite{MP2, PI3}.  Note that these questions are related: one of the two known proofs of the global integrability for finite-dimensional Lie algebras uses the Ado theorem in an essential way.

Still, several classes of Sabinin algebras are similar to the Lie algebras in their integrability properties.  In particular, each finite-dimensional Malcev algebra is a tangent algebra of a unique simply-connected global Moufang loop \cite{Kerd, Na}.  Moreover,  there is a version of the Ado theorem for Malcev algebras \cite{ShPI1}. The main purpose of the present note is to point out that that the same holds for  finite-dimensional nilpotent Sabinin algebras: they can be integrated to globally defined nilpotent loops and can be embedded into finite-dimensional non-associative algebras in a way consistent with their operations. Note that, in contrast to Lie and Malcev algebras, the structure of a nilpotent Sabinin algebra involves, in general, more than one operation, although the number of these operations is necessarily finite.

Each Sabinin algebra can be realized as a subspace in a Lie algebra, called the Lie envelope of the Sabinin algebra. We give a new set of axioms for Sabinin algebras 
which reflect this fact and discuss the construction of Lie envelopes in general. We also show that the nilpotency of a Sabinin algebra is in direct relation with the existence of a nilpotent Lie envelope. 

We finish this note by observing that for any torsion-free nilpotent loop $L$ there is  a non-associative algebra whose invertible elements form a loop containing $L$. This is a direct consequence of the non-associative generalization of the Jennings theorem \cite{M2}. This statement, whose associative prototype can be found in  \cite[Proposition 3.6]{Quillen}, is a non-linear version of the Ado theorem. In a way, it is more fundamental than the Ado theorem for Sabinin algebras since it requires no knowledge of the primitive operations.

\subsection*{On the notation and terminology} 

We shall say that a loop is {\em nilpotent} if its commutator-associator filtration, as defined in \cite{M1}, has finite length. This notion of nilpotency is different from nilpotency in the sense of Bruck~\cite{Bruck}. While nilpotent loops as defined here are nilpotent in the sense of Bruck, the converse is not necessarily true. The use of non-standard terminology is amply justified by the fact that many basic results from the theory of nilpotent groups extend to the non-associative setting if nilpotency is understood in terms of the commutator-associator filtration rather than Bruck's lower central series. For details see \cite{M1, M2, MP1}.

\medskip

We refer to \cite{MS2, PI2, MS1} for the definition and basic properties of Sabinin algebras. Here we shall only mention that  a Sabinin algebra is a vector space with two infinite sets of multilinear operations on it. The first set of operations are the {\em Mikheev-Sabinin brackets} $\langle x_1, \ldots, x_n ; y, z\rangle$ which are defined for each $n\geq 0$. Informally, they are abstract versions of the covariant derivatives $\nabla_{x_1}\ldots\nabla_{x_n} T(y,z)$ of the torsion tensor of a flat affine connection, namely, the canonical connection of a loop. Mikheev-Sabinin brackets satisfy certain identities which can be interpreted as the Bianchi identities and their derivatives. We shall not need their precise form here. 

The second set of operations, $\Phi(x_1, \ldots, x_n; y_1, \ldots, y_m)$, are defined for all $n\geq 1, m\geq 2$ and are known collectively as the {\em multioperator}. The only kind of identities they satisfy is the total symmetry in each of the two groups of variables. The multioperator has its origin in the Taylor series of the function that measures the failure of a loop to be right alternative. We stress that the two sets of operations are completely independent of each other. In particular, for each Sabinin algebra there exists a Sabinin algebra with the same Mikheev-Sabinin brackets and the trivial multioperator. We shall call Sabinin algebras with the trivial multioperator {\em flat} Sabinin algebras. 

In general, by a {\em bracket} in a Sabinin algebra we shall understand a multilinear operation which can be obtained from Mikheev-Sabinin brackets and  the multioperator by composition. The {\em weight} of a bracket is the number of its arguments. A Sabinin algebra  is {\em nilpotent of class $n$} if $n$ is the smallest integer such that all the brackets of weight  $n+1$ and greater are identically zero in it. Explicitly, given a Sabinin algebra $\ll$,
define $\ll_1=\ll$ and let $\ll_n$ to be the  subspace of $\ll$ spanned by 
\begin{itemize}
\item
$\langle \ll_{i_1}, \ldots, \ll_{i_l}; \ll_i, \ll_j\rangle$ with $i_1+\ldots+i_l+i+j\geq n$;
\item
$\Phi(\ll_{i_1}, \ldots, \ll_{i_p}; \ll_{j_1}, \ldots, \ll_{j_q})$ with $i_1+\ldots+i_p+j_1+\ldots+j_q\geq n$.
\end{itemize}
Then $\ll$ is nilpotent of class $n$ if $\ll_n\neq \ll_{n+1}=0$.

\medskip

Finally, recall that for each Sabinin algebra $\ll$ there exists a universal enveloping algebra $U(\ll)$ which is a cocommutative and coassociative but not necessarily associative Hopf algebra into which $\ll$ is embedded as the space of primitive elements. The operations on $\ll$ are induced by certain algebraic operations in $U(\ll)$ in the same manner as the Lie bracket in a Lie algebra is induced by the commutator in its universal enveloping algebra. Denote by $\underline{u}$ the sequence $x_1, \ldots, x_n$ of elements of $\ll$ and let $u$ be the left-normed product $((x_1x_2)\ldots)x_n$. Similarly, write  $\underline{v}$ for the sequence $y_1,\ldots, y_m$ and $v$ for the corresponding product. The Shestakov-Umirbaev primitive operations $p(\underline{u}; \underline{v}; z)$ are defined inductively by the identity
\begin{equation}\label{defShU}
(uv)z-u(vz)= u_{(1)}v_{(1)}\cdot p(\underline{u}_{(2)}; \underline{v}_{(2)}; z),
\end{equation}
where the Sweedler notation is used. Each $p(\underline{u}; \underline{v}; z)$ induces a bracket of weight $m+n+1$ on $\ll$. 
The Mikheev-Sabinin brackets on $\ll$ can be obtained as

\begin{equation}
\begin{gathered}\label{defMS}
\langle y, z\rangle= -yz+zy,\\
\langle x_1, \ldots, x_n ; y, z\rangle= -p(x_1, \ldots, x_n; y ; z) + p(x_1, \ldots x_n; z ; y)
\end{gathered}
\end{equation}
and the multioperator as 
$$
\Phi(x_1, \ldots, x_m; y_1, \ldots, y_n)\\
=\frac{1}{m!}\frac{1}{n!}\sum_{\sigma\in S_m, \tau\in S_n} p(x_{\sigma(1)}, \ldots, x_{\sigma(m)};
y_{\tau(1)}, \ldots, y_{\tau(n-1)}; y_{\tau(n)}),
$$
where $S_k$ is the symmetric group on $k$ letters.
With these definitions any algebra becomes a Sabinin algebra;  we denote the Sabinin algebra structure on an algebra $A$ by $\SU(A)$.

\subsection*{Acknowledgments}
The authors were supported by the following grants:  FAPESP 2012/21938-4 and CONACyT grant 168093-F (J.M.), FAPESP 2012/22537-3 (J.M.P.-I.), CNPq 3305344/2009-9 and FAPESP 2010/50347-9 (I.P.Sh.), Spanish Government project MTM 2010-18370-C04-03 (all three). The first two authors would like to thank the Institute of Mathematics and Statistics of the University of S\~ao Paulo for hospitality.

\section{Formal loops on filtered vector spaces}
Let $V=V_1\supseteq V_2 \supseteq\ldots $ be a filtered vector space. Let us say that a multilinear operation $q: V^{\otimes k}\to V$ respects the filtration on $V$ if $q(x_1, \ldots, x_k)\in V_m$ whenever $x_i\in V_{j_i}$ with $j_1+\ldots+j_k=m$. Given such an operation $q$, a generalized monomial in $x$ and $y$ is the map $V\otimes V\to V$ obtained by substituting two variables $x,y$ as arguments into $q$. Similarly, one defines generalized monomials in more than two variables.
A generalized monomial is of weight $n$ if its image is contained in $V_n$ but not in $V_{n+1}$.  The weight of a monomial is at least as big as the number of arguments of the corresponding multilinear operation, that is, its degree. 

Let $F$ be a unital formal multiplication on $V$ 
$$F (x,y)=x+y+\sum_{n, \alpha} P_{n, \alpha} (x, y),$$
where each $P_{n, \alpha}$ is a generalized monomial of weight $n>1$ which involves both arguments, and $\alpha=\alpha(n)$ varies over some index set  which is finite for each $n$.

\medskip

\begin{theorem}~\label{the}
Assume that $\bigcap_i V_i=0$. Then $V$ is a loop with the product $F$ and the subspaces $V_i$ form an $N$-sequence. 
\end{theorem}
For the definition of $N$-sequences in the non-associative context we refer to  \cite{M2}.

An immediate corollary is that $V$ is residually nilpotent. Moreover, if $V_{n+1}=0$ for some $n$, then $V$ is nilpotent of class at most $n$.
Note that it is not necessarily true that the $N$-sequence $V_i$ coincides with the commutator-associator filtration on $V$.
\begin{proof}
Write the formal loop $F$ as
$$F(x,y)=x+y+H(x,y).$$
Let $D(x,y)$ be the infinite expression
$$D(x,y)= -x+y-H(x, -x+y-H(x, -x+y-H(x, -x+y-\ldots ))).$$ 
Since $H(x,y)$ is a linear combination of generalized monomials in $x$ and $y$, the expression $D(x,y)$ is, actually, itself an infinite linear combination of generalized monomials, with a finite number of monomials of each weight. Moreover, it is easy to see that
$$F(x, D(x,y))=y.$$
Similarly, there exists an infinite linear combination 
$$D'(x,y)= x-y-H(x-y-H(x-y-H(x-y-\ldots ,y),y),y).$$ 
such that 
$$F(D'(x,y), y)= x.$$

The formal series $F$ converges on $V$ since modulo each $V_k$ it is a finite sum and $\bigcap_i V_i=0$. The product it defines is a loop: the left and right divisions are given by $D$ and $D'$ respectively, and the unit is $0\in V$.  We denote this loop by $V^F$.

The subspaces  $V_k$ with this product become subloops $V_k^F\subset V^F$. We need to show that these subloops form an $N$-sequence on $V^F$. For this it is sufficient to establish that the commutators, associators and associator deviations can be written in $V^F$  as linear combinations of generalized monomials which involve each of the arguments of the corresponding loop operation at least once.

The loop associator $(x,y,z)^F$ can be written as 
\begin{equation}\label{associator}
(x,y,z)^F = D(F(x, F(y,z)), F(F(x,y),z)).
\end{equation}
We have  
\begin{multline*}
F(x, F(y,z))=x+y+z+H(x,z)+H(x, y+z+ H(y,z))\\
=x+y+z+H(x,z)+H(x,y)+H(y,z)+\text{monomials involving } x,y \text{ and } z,
\end{multline*}
and the same expression for $F(F(x,y),z)$, though, of course, the linear combinations of monomials involving $x,y$ and $z$ will be different. Using the expression for $D$ we see that $(x,y,z)^F$ is a combination of monomials which involve all the arguments.

Now, let $\phi$ be a formal linear combination of generalized monomials in $m$ different arguments $x_1,\ldots, x_m$ with the property that each monomial involves all the $x_i$. Such $\phi$ gives rise to a map $V^m\to V$; as we have just seen, the loop associator is an example of such a map. Fix $j$ between 1 and $m$. Comparing the two expressions 
$$r=\phi(x_1,\ldots,x_{j-1}, F( x_j, x_{j+1} ), x_{j+2},\ldots, x_m)$$
and
$$s=F( \phi(x_1,\ldots,x_{j-1}, x_j, x_{j+2},\ldots, x_m),   \phi(x_1,\ldots,x_{j-1}, x_{j+1}, x_{j+2},\ldots, x_m) )$$
we see that the terms that do not contain at least one of the $x_i$ for $i=1,\ldots, m+1$ coincide in both of them. This implies that $D(s, r)$ consists of brackets which involve each of the arguments $x_1, \ldots, x_{m+1}$ at least once. Note that taking $\phi$ to be the expression (\ref{associator}) for the associator, we obtain in this way all associator deviations. Therefore, we have proved that the associator and all the associator deviations respect the $V_k^F$. The argument for the commutator is entirely similar.
\end{proof}

There are several natural examples of loops that fall within the premises of Theorem~\ref{the}.
\subsection*{Power series under composition}
Let $A$ be an associative unital algebra and $B(A)$ be the set of power series in one variable $x$ with coefficients in $A$ which have the form
$$a(x) = x+a_1 x^2 + a_2 x^3 +\ldots $$
Assuming that $x$ commutes and associates with the elements of $A$, one defines the composition of $a, b\in B(A)$ as the result of substitution of $b$ into $a$ instead of $x$. This endows  $B(A)$ with the structure of a loop, which is not associative unless $A$ is commutative. The $n$th coefficient of the composition $a(x) \circ b(x)$ is readily seen to be
$$\sum_{\begin{array}{c} m+i_1+\ldots+i_k=n,\\ k\leq m+1\end{array}} a_m b_{i_1}\ldots b_{i_k},$$
where $a_0=b_0=1$.
 
The space $B(A)$ is naturally filtered by
$$B(A)_i=\{x+a_i x^{i+1}+ a_{i+1} x^{i+2}+\ldots\}.$$
The map that truncates a power series after the first non-trivial coefficient respects the filtration. Since the product in $A$ is bilinear, and  we see that Theorem~\ref{the} can be applied to $B(A)$; in particular, $B(A)$ is residually nilpotent. 

The graded Sabinin algebra associated with the filtration by the degree on $B(A)$ is the sum of a copy of $A$ in each degree. The low-degree operations can be calculated explicitly. The commutator induces the bilinear bracket
$$[a_i, b_j] = i\cdot a_i b_j - j\cdot b_j a_i,$$
and the associator gives the trilinear operation
$$(a_i, b_j, c_k) = \frac{i(i+1)}{2}\cdot a_i (b_jc_k-c_k b_j).$$
Here, of course, $\deg [a_i, b_j] =i+j$ and $\deg (a_i, b_j, c_k) =i+j+k$.
\begin{remark}
In the case when $A$ is commutative, the loop $B(A)$ is a group. Groups of this form have been studied in some detail; see, for instance \cite{Babenko, BFK, Camina, Jennings}. Especially interesting is the case $A=\mathbb{F}_p$ which gives the so-called Nottingham group.
\end{remark}

\subsection*{Non-associative power series}
Let $A$ be a unital algebra and $C(A)$ be the set of power series in one non-associative variable $x$ with coefficients in $A$ which have the form
$$a(x) = 1+a_x x + a_{x^2} x^2 + a_{x^2 x} x^2 x +a_{x x^2} x x^2 +\ldots $$
Assume that $x$ commutes and associates with the elements of $A$; then $C(A)$ is a loop with respect to the multiplication of the power series. The loop $C(A)$ is non-associative even when $A$ is commutative. The loop $C(A)$ can be filtered by setting 
$$C(A)_i=1+\sum_{\deg \tau \geq i} a_\tau \tau,$$
where $\tau$ are non-associative monomials in $x$. Again, since the product in $A$ is bilinear, the premises of Theorem~\ref{the} are satisfied.

When $A=\Q$ the Sabinin algebra associated with this filtration is the Shestakov-Umirbaev Sabinin algebra of the free non-associative algebra on one generator.

\begin{remark}
The loop $C(\Q)$ contains an interesting subloop, which consists of all series of exponential type: these are the power series whose composition with any primitive element in the Hopf algebra of non-associative power series with rational coefficients is group-like. The corresponding Sabinin algebra, with respect to the filtration by the degree, is free since the variety of Sabinin algebras is Schreier \cite{Chi}.
\end{remark}

\section{Integration of nilpotent Sabinin algebras}
Here we shall prove that each nilpotent Sabinin algebra is a tangent algebra of a unique nilpotent simply-connected loop of the same class. 
The construction that we use to establish this fact gives an equivalence of the category of  nilpotent Sabinin algebras of class $n$ and simply-connected nilpotent  loops of class $n$. 

\subsection{Primitive series}
Recall that the space of primitive elements in  the free algebra on $x$ and $y$ is the free Sabinin algebra on $x$ and $y$.
A {\em primitive series} is a (possibly infinite) linear combination of primitive elements in the free algebra on $x$ and $y$ which is of the form 
$$F(x,y)=x+y+\text{brackets that involve both } x \text{ and } y.$$
Theorem~\ref{the} shows that any primitive series defines a functor from nilpotent Sabinin algebras to nilpotent loops of, at most, the same class. We call this functor  an {\em integration} if it is inverse to taking the tangent algebra of a loop. In the case of Lie groups, such an integration is provided by the  Baker-Campbell-Hausdorff series.

\subsection{The Baker-Campbell-Hausdorff series}
Recall that  the exponential power series $\exp$ and the logarithmic power series $\log$ interchange the primitive and the group-like elements in the algebra of non-commutative  formal power series in one or several variables. As a consequence, the expression 
$$\mathcal{C}(x,y)=\log(\exp x \exp y)$$ in the algebra of non-commutative  formal power series in $x$ and $y$ can be written in terms of the iterated commutators in $x$ and $y$. In particular, given a nilpotent Lie bracket on a vector space, the expression $\mathcal{C}(\cdot, \cdot)$, with the commutators understood as Lie brackets, defines an associative product. Since $\mathcal{C}(\cdot,\cdot)$ is a primitive series, this product produces a nilpotent group, which can be shown to integrate the Lie algebra.

Similar series can also be defined in the non-associative situation. Indeed, group-like elements in a complete non-associative Hopf algebra form a loop. If $\exp$ is a non-associative power series in one variable that sends primitive elements bijectively onto the the group-like elements in the algebra of non-associative power series in one variable, the same reasoning as before can be used to show that  the expression $\log(\exp x \exp y)$ can be written as an infinite linear combination of primitive operations in $x$ and $y$ and, hence, gives a primitive series. A primitive series of this type was studied in \cite{GH}.

It is not immediately clear, however, that the primitive series defined in this way gives an integration functor.  An interesting point is that the power series $\exp$ and $\log$ are not uniquely determined by the condition that they interchange primitive and group-like elements \cite{GH}. Our strategy will be to use a different primitive series whose definition is based on the geometry of the Mikheev-Sabinin brackets.

\medskip

Let $\overline{\A}$ be the augmentation ideal of the algebra $\A$ of non-associative power series in several variables  $x_1, \ldots, x_n$. 
Then $$\F=1+\overline{\A}$$ is an infinite-dimensional loop whose tangent space $T_a$ at a point $a$ can be identified with $a+\overline{\A}$. On the tangent bundle to $\F$ there exists a canonical flat connection $\nabla$ whose associated parallel transport sends the $1+x \in T_1$ to $a+ax\in T_a$. It follows that the parallel transport from $b$ to $a$ sends $b+x\in T_b$ to $a+a(b\backslash x)\in T_a$.

The geodesics and the exponential map for $\nabla$ can be calculated explicitly  \cite[Appendix]{MP3}. For the exponential map at $1$ we have
$$\exp x = 1+x+\frac{x^2}{2!}+\frac{x^2x}{3!}+\frac{(x^2x)x}{4!}+\ldots$$
for all $1+x\in T_1$. There are also explicit expressions for the coefficients of the corresponding logarithm, but we shall not need them. Note that $\exp$ sends the space of the primitive elements in $\A$ to the subloop of the group-like elements in $\F$. The exponential map of $\nabla$ at an arbitrary $b\in \F$ is given by
$$\exp_a x = a+ x+\frac{x(a\backslash x)}{2!}+\frac{(x(a \backslash x))(a \backslash x)}{3!}+\frac{((x(a \backslash x))(a \backslash x))(a \backslash x)}{4!}+\ldots.$$

Recall that with each (local) loop one can associate a (local) right alternative loop, namely, the geodesic loop of the canonical connection. 
In the present context, the formula for the right alternative product $\times$ on the loop $\F$ is 
$$a\times b = \exp_a( a\log b);$$
this product also defines a global loop on $\F$. If $a=\exp x$ and $b =\exp y$ it is easy to see that
$$a\times b = \sum_{i,j=0}^{\infty}\frac{x^i y^j}{i!j!}$$
where in each term $x^i y^j$ the parentheses are assumed to be left-normed. 

The subloop of all the group-like elements in $\F$ remains a subloop under the right alternative product. This is is immediate from the definition of the canonical connection, but is also clear form the above expression for $a\times b$. It follows that $\log (\exp x\times \exp y)$ can be expressed as a series consisting of compositions of the Mikheev-Sabinin operations $\langle \cdot, \ldots, \cdot ; \cdot, \cdot\rangle$ in the arguments $x$ and $y$. This is what we call the {\em right alternative Baker-Campbell-Hausdorff series}. 

\subsection{The Mikheev-Sabinin brackets}
In order to show that the right alternative Baker-Campbell-Hausdorff  series gives an integration of a flat nilpotent Sabinin algebra we do not need any explicit formula for this series. It is sufficient to observe that, by construction, the canonical connection of the Baker-Campbell-Hausdorff  product $\times$ is the same as that of the loop $\F$. 

\begin{proposition}
Let $\ll$ be a flat nilpotent Sabinin algebra of class $n$. The right alternative BCH series defines a class $n$ nilpotent right alternative loop on $\ll$, whose tangent Sabinin algebra  coincides with $\ll$.
\end{proposition}
\begin{proof}
Let $\Prim{\A}$ be the subspace of primitive elements in $\A$. It is the completion of the free Sabinin algebra on $x_1, \ldots, x_n$ with the operations coming form the primitive operations of Shestakov and Umirbaev. This Sabinin algebra structure agrees with the one it acquires as a subspace of the tangent space to $\F$ (see, for instance, \cite[Theorem 6.1]{MP3}). Since both the original product and the BCH product on $\F$ share the same canonical connection, the Mikheev-Sabinin brackets of the Sabinin algebra tangent to the BCH product on $\Prim{\A}$ coincide with the corresponding brackets in $\Prim {\A}$.

Given a nilpotent Sabinin algebra $\ll$, a  homomorphism 
$$\Prim{\A}\to\ll$$ induces a homomorphism of the corresponding BCH loops. The construction of the tangent Sabinin algebra of a loop is functorial. The same is true for the canonical connection and the corresponding torsion tensor and its derivatives. As a consequence, the tangent Sabinin algebra of a BCH loop for a nilpotent Sabinin algebra $\ll$ coincides with $\ll$.
\end{proof}

It may be instructive to calculate the Mikheev-Sabinin brackets for the BCH product on  $\Prim\A$ explicitly.  
First, note that if a vector field
$$a\mapsto x(a),\quad a\in\F$$ is parallel with respect to the connection $\nabla$ and $x(1)=x$, we have that $x(a)=ax$.
The Lie bracket of two parallel vector fields is calculated as
$$[ay, az]=(ay)z-(az)y.$$
Indeed, let $f_\mu$ be the function on $\F$ that assigns to a power series the coefficient of the monomial $\mu$ in it.
In order to define a vector field on $\F$ it is sufficient to specify its action on the functions of the form $f_\mu$ for all $\mu$. We have that $$az(f_\mu)=f_\mu(az)$$ and $$ay(f_\mu(az))=f_\mu((ay)z).$$
Now, by the definition of the torsion tensor we have that at a point $a$ $$T(ay, az)=-[ay, az]=(az)y-(ay)z.$$
For all $n\geq 1$ set 
$$P(x_1, \ldots, x_n; y, z)=\nabla_{ax_1}\ldots\nabla_{ax_n} ((ay)z).$$ 
Let
$$u=((x_1x_2)\ldots)x_n,$$
$$^au=(((ax_1)x_2)\ldots)x_n,$$
and write $\underline{u}={^a\underline{u}}$ for the sequence $x_1, \ldots, x_n$. 

From the definition of the covariant derivative, and using induction, we get the following identity
\begin{equation}\label{nonlinearshu}
(^au y) z-{^au}\cdot  a\backslash (ay\cdot z)=  {^au}_{(1)}\cdot a\backslash P( ^a\underline{u}_{(2)}; y,z).
\end{equation}
For instance, for $n=1$ it reads
$$(ax\cdot y) x- ax \cdot a\backslash (ay\cdot z)= P(x;y,z).$$
When $a=1$ formula (\ref{nonlinearshu}) becomes the identity defining the Shestakov-Umirbaev operations $p(\underline{u};y;z)$, and we see that the Mikheev-Sabinin brackets on the primitive elements of $\A$ defined with the help of the BCH product coincide with the brackets induced by the primitive operations  $p(\underline{u};y;z)$.

\subsection{The general Baker-Campbell-Hausdorff  series} The case of a non-trivial multioperator can be handled in the following manner. 
Let $\widehat{\Phi}(x,y)$ be the following element in the completion of the free Sabinin algebra on $x$ and $y$:
$$\widehat{\Phi}(x,y)=y+\Phi(x; y,y)+ \Phi(x; y,y,y)+\Phi(x,x; y,y)+\ldots.$$
Apart from the first term, this is just the sum of all the operations of the multioperator.
Let $\mathcal{C}(x, y)$ be the right alternative Baker-Campbell-Hausdorff  series.
\begin{proposition}
The expression $\mathcal{C}(x, \widehat{\Phi}(x,y))$ is a primitive series which defines on an arbitrary nilpotent Sabinin algebra the structure of the nilpotent loop that integrates it.
\end{proposition}

This follows directly from the definition of the multioperator for a smooth local loop, see \cite{MP3, MS1}.

\subsection{The uniqueness of the integration}\label{here-we-define-B}
\begin{proposition}
For a given Sabinin algebra,  the only simply-connected analytic loop that integrates it is the one specified by the Baker-Campbell-Hausdorff  formula.
\end{proposition}
\begin{proof}
Let $\ll$ be a Sabinin algebra, $L$ an analytic loop integrating $\ll$, and $B(\ll)$ the loop given by the Baker-Campbell-Hausdorff  formula. The exponential map
$\ll\to L$ can be thought of as a map $E: B(\ll)\to L$. Since the infinitesimal loop integrating $\ll$ is unique, $E$ is locally a loop homomorphism.
As both loops together with the exponential map are analytic, the map $E$, actually, a bona fide homomorphism. In particular, it is a covering, and the only possibility for $L$ to be simply connected is if $E$ is a loop isomorphism.
\end{proof}
\begin{remark}
If a Sabinin algebra integrates to an analytic loop $L$, there is always an infinite number of non-analytic loops which integrate the same Sabinin algebra. For instance, consider a perturbation of the product in $L$ which is trivial on a neighbourhood of $L\times \{e\} \cup \{e\}\times L$. If it is small enough, it will give rise to a loop with the same Sabinin algebra. 
It can be shown that nilpotent loops are always analytic.
\end{remark}
The results of this section can be summarized in the following statement:

\begin{theorem}
The Baker-Campbell-Hausdorff series defines an equivalence of the category of nilpotent Sabinin algebras of class $n$ and the category of simply connected nilpotent loops of class $n$.
\end{theorem}

\section{The Ado theorem for nilpotent Sabinin algebras} 
The associative case of the following theorem and of its non-linear version in the next section appears in \cite[Proposition 3.6]{Quillen}.
\begin{theorem}\label{thm:ado} 
Let $\ll$ be a nilpotent Sabinin algebra of class $n$, $U(\ll)$ its universal enveloping algebra and $\overline{U(\ll)}\subset U(\ll)$ the augmentation ideal. Then the composition
$$\ll\hookrightarrow U(\ll)\to U(\ll)/ {\overline{U(\ll)}}\,^{n+1}$$
is injective. 
\end{theorem}
In particular, any nilpotent Sabinin algebra of finite dimension can be realized as a Sabinin subalgebra in a finite-dimensional algebra, since ${\overline{U(\ll)}}\,^{n+1}$ is of finite codimension in $U(\ll)$ by the Poincar\'e-Birkhoff-Witt theorem.

\medskip

Let $\B_n=\emptyset$ and let $\B_n\sqcup\ldots\sqcup\B_i$ be a basis of $\ll_i$. 
Define $N(a)=i$ if $a\in \B_i$ and set $N(0)=\infty$. Choose an order on  $\sqcup\,\B_i$ so that $N$ is non-decreasing. The corresponding  PBW basis consists of products of the form
$$((a_1a_2)\ldots) a_d$$
with $a_k\leq a_{k+1}$ and $d\geq 0$. We define 
$$N(((a_1a_2)\ldots) a_d)=N(a_1)+\ldots + N(a_d)$$
and extend this definition to the whole of $U(\ll)$ by
$$N\left(\sum \alpha_I a_I\right)= \min \{ N(a_I)\, |\, \alpha_I\neq 0 \},$$
where $I$ varies over all the ordered sets of elements of $\sqcup\,\B_i$ and $a_I$ is the basis element of $U(\ll)$ corresponding to $I$. 
Clearly,
$$N(u+v)\geq \min\{N(u), N(v)\}.$$

\begin{lemma}\label{thelemma}
$N(uv)\geq N(u)+N(v).$
\end{lemma}
\begin{proof}

It is sufficient to prove the lemma in the case when $u$ and $v$ are two elements of the PBW basis. We shall prove a slightly more precise statement, namely that if $$I=A\sqcup B$$ is  an ordered set of basis elements of $\ll$, then $$a_I-a_{A}a_{B}$$ is a linear combination of basis elements of $U(\ll)$ which have smaller length than $a_I$ and the same or greater value of $N$ as $a_I$.

First, consider the case where $B=\{a\}$ consists of one basis element of $\ll$, and let $$a_A=((a_1a_2)\ldots) a_n=a_{A'} a_n.$$
We shall use induction on $n=|\,A\,|.$
If $a\geq a_n$, $a_A a$ is an element of the PBW basis and there is nothing to prove.
Otherwise, observe that (\ref{defShU}) and (\ref{defMS}) imply that for all $w\in U(\ll)$ and $y,z\in \ll$
$$wy\cdot z - wz\cdot y= - w_{(1)}\langle \underline{w}_{(2)}; y, z\rangle,$$
and, in particular,
$$a_Aa- a_{A'} a\cdot a_n = - a_{A'_{(1)}}\langle  \underline{a}_{A'_{(2)}}; a_n, a\rangle.$$
By the induction assumption the right-hand side is a linear combination of basis elements with length smaller than $n+1$ (actually, smaller than $n$). Also, the value of $N$ on the right-hand side  is equal to at least 
$$N(a_{A'_{1}})+ N(a_{A'_{2}})+N(a_n)+N(a)= N(a_A)+N(a),$$
where $A'=A'_1\sqcup A'_2$ is an arbitrary decomposition of the ordered set $A'$.
If the term $a_{A'} a\cdot a_n$ is an element of the PBW basis, it is equal to $a_I$ with $I=A\sqcup\{a\}$, and we are done. If not, applying the induction assumption once more we get 
$$a_{A'} a= a_{I'}+\sum\alpha_J a_J,$$
where $I' = A' \sqcup \{a\}$ and $a_J$ are basis elements with length smaller than $n=N(a_{I'})$ and such that $N(a_J)\geq N(a_{I'})$. Since $a_n\geq a_k$ for all $a_k\in I'$, the product $a_{I'}a_n$ is a basis element, namely, $a_{I'}a_n=a_{A\sqcup \{a\}}$. On the other hand, again by the induction assumption, the length of any element in the decomposition of $a_J a_n$ is smaller than that of $a_{A\sqcup\{a\}}$ and $$N(a_Ja_n)\geq N(a_J)+N(a_n) \geq N(a_{A'\sqcup \{a\}})+ N(a_n)=N(a_{A\sqcup \{a\}}).$$ This settles the case  $|\,B\,|=1$; in other words, we have shown that $$N(uv)\geq N(u)+N(v)$$ whenever $v\in \ll$.

In the case when $|\,B\,|>1$ we use induction on $m=|\,B\,|$.
Write  $$a_B=((a_1a_2)\ldots) a_m=a_{B'} a_m.$$
We have $$a_A a_B= a_A \cdot a_{B'}a_m= a_A  a_{B'} \cdot a_m + a_{A_{(1)}}  a_{{B'}_{(1)}}\cdot p( \underline{a}_{A_{(2)}} ; a_{{B'}_{(2)}}; a_m).$$
It remains to use the induction assumption.

\end{proof}

\begin{corollary} If $\ll_{n+1}=0$, the ideal ${\overline{U(\ll)}}\,^{n+1}$  satisfies
${\overline{U(\ll)}}\,^{n+1}\cap \ll=0$.
\end{corollary}
\begin{proof}
According to Lemma~\ref{thelemma}, the value of $N$ on any element of ${\overline{U(\ll)}}\,^{n+1}$ is greater than $n$. On the other hand,  $a\in \ll$ with $N(a)>n$ must belong to $\ll_{n+1}=0$.
\end{proof}

This proves Theorem~\ref{thm:ado}.

\section{Lie envelopes and axioms for Sabinin algebras}

\subsection{Lie envelopes of flat Sabinin algebras}
Let $\lie$ be a Lie algebra with a vector space decomposition
$$\lie=\subalg\oplus \ll$$
where $\subalg$ is a subalgebra. Denote by $\pi$ the
projection of $\lie$ to $\ll$ along $\subalg$. Then the iterated Lie brackets in
$\lie$ give rise to a family of multilinear operations on $\ll$:
\begin{equation}\label{pi}
(x_1,\ldots,x_n)=\pi [x_1,[\ldots [x_{n-1},x_n]]]
\end{equation} 
for each $n\geq 2$ and $x_i\in V$. One may also define the family of multilinear brackets $\langle x_1, \ldots, x_n ; y, z\rangle$ by the identity
\begin{equation}\label{brackets}
(x_1,\ldots,x_n, y, z) + \langle x_1, \ldots, x_n; y, z\rangle \\
+\sum_{t=1}^n \sum_{\alpha} (x_{\alpha_1},\ldots x_{\alpha_t}, \langle x_{\alpha_{t+1}}, \ldots, x_{\alpha_n}; y, z\rangle )=0,
\end{equation} 
where $\alpha$ are the shuffles of the $x_i$, that is, permutations with $\alpha_1<\ldots<\alpha_t$ and $\alpha_{t+1}<\ldots<\alpha_n$. 
It turns out that the brackets defined in this way satisfy the same identities as the Mikheev-Sabinin brackets; as a consequence, $\ll$ receives the structure of a flat Sabinin algebra, see \cite{MS1, PI2}. The Lie algebra $\lie$ is known as a {\em Lie envelope} for $\ll$. 

\begin{theorem}\label{nilenvelope}
A flat Sabinin algebra is nilpotent if and only if it has a nilpotent Lie envelope of the same class.
\end{theorem}

We defer the proof to the Section~\ref{SavLa}.

\subsection{Iterated brackets in a free Lie algebra} Our goal is to give a set of axioms for Sabinin algebras in terms of the operations $(x_1,\ldots,x_n)$. 
For this we need one fact about the relations in the free Lie algebra.

If $[x_1,x_2]$ is a Lie bracket, we denote by $[x_1,\ldots,x_n]$ the iterated
right-normed Lie bracket $[x_1,[\ldots [x_{n-1},x_n]]]$.
Using the Jacobi identity (in this form it is also known as left Leibniz identity)
\begin{equation}\label{eq:jacobi-0}
[[x,y],z]=[x,[y,z]]-[y,[x,z]]
\end{equation}
any composition of Lie brackets can be expressed (``re-written'') via the right-normed iterated brackets. It is not hard to prove that such expression is unique; see  \cite[Lemma~1.3]{LP}.

Now, write
\begin{equation}\label{eq:jacobi}
[[x_1,\ldots,x_n],y]=[x_1,[[x_2,\ldots,x_n],y]]-
[[x_2,\ldots,x_n],[x_1,y]].
\end{equation}
Therefore,
\begin{equation}\label{eq:rewrite}
[[x_1,\ldots,x_n],y]=\sum_{\alpha}(-1)^{h(\alpha)}
[x_{\alpha_1},\ldots x_{\alpha_n}, y],
\end{equation}
where the summation is taken over all permutations $\alpha$ of the
set $\{1,\ldots,n\}$ for which there exists $s$ such that
$\alpha_i<\alpha_{i+1}$ for $i<s$, and $\alpha_i>\alpha_{i+1}$ for
$i\geq s$, and $h(\alpha)=n-s$. Clearly, for any such $\alpha$ we
have that $\alpha_s=n$.

Taking $y=[y_1,\ldots,y_m]$ we can re-write (\ref{eq:rewrite}) as
\begin{equation}\label{eq:rewrite2}
[[x_1,\ldots,x_n],[y_1,\ldots,y_m]]=\sum_{\alpha}(-1)^{h(\alpha)}
[x_{\alpha_1},\ldots x_{\alpha_n}, y_1,\ldots,y_m].
\end{equation}

One may ask what are the universal relations that hold among the
right-normed brackets. More precisely, let $\bigL$ be the free Lie
algebra in the generators $a_1,\ldots,a_n$. If some relation holds
in $\bigL$, then, expressing each term of this relation via right-normed
brackets we obtain a relation among the brackets $[a_{q_1},\ldots,
a_{q_r}]$.  In this sense, any relation among the brackets
$[a_{q_1},\ldots, a_{q_r}]$ is a consequence of the antisymmetry
and the Jacobi identity in $\bigL$.
\begin{lemma}
Any relation among the right-normed brackets $[a_{q_1},\ldots,
a_{q_r}]$ in a free Lie algebra is a consequence of antisymmetry relations in
$\bigL$ only.
\end{lemma}
\begin{proof}

We need to show that the Jacobi identity (\ref{eq:jacobi-0}) by
re-writing transforms into the trivial identity $0=0$. This follows
from the uniqueness of the re-writing since the right-hand side of
(\ref{eq:jacobi-0}) is the re-writing of its left-hand side.

\end{proof}

\subsection{Sabinin algebras via Lie algebras}\label{SavLa}

Let $\lie$ be a Lie algebra, $\subalg\subset \lie$ --- a subalgebra and $\ll$
--- a vector subspace of $\lie$ such that $\lie=\subalg\oplus \ll$, with
$\pi$ the projection of $\lie$ to $\ll$ along $\subalg$ and
$$(x_1,\ldots,x_n)=\pi [x_1, \ldots ,x_n].$$

There are two kinds of relations satisfied by these operations. The
relations of the first kind reflect the fact that these operations
come from Lie brackets. Namely, the antisymmetry of the Lie bracket
gives
\begin{equation}\label{eq:one}
(x_{1},x_2)+(x_2,x_{1})=0,
\end{equation}
and the Jacobi identity translates into the following identity:
\begin{equation}\label{eq:two}
(x_{1},x_{2},x_3)+(x_{2},x_{3},x_{1})+ (x_{3},x_{1},x_{2})=0.
\end{equation}
The identities of the second kind express the fact that $\subalg$ is a
subalgebra and not just a vector subspace. If $u,v\in \lie$ are
arbitrary elements, then
$$\pi [\pi u -u, \pi v-v]=0.$$
Setting $u=[x_1,\ldots,x_n]$ and $v=[y_1,\ldots,y_m]$  we get
\begin{multline*}
\pi[[x_1,\ldots,x_n],[y_1,\ldots, y_m]]
+((x_1,\ldots,x_n),(y_1,\ldots,y_m) )\\
=((x_1,\ldots,x_n),\, y_1,\ldots,y_m)-((y_1,\ldots,y_m),\,
x_1,\ldots,x_n).
\end{multline*}
Using formula~(\ref{eq:rewrite2}), we obtain
\begin{multline}\label{eq:three}
-((x_1,\ldots,x_n),(y_1,\ldots,y_m) )+((x_1,\ldots,x_n),\,
y_1,\ldots,y_m)-((y_1,\ldots,y_m),\,
x_1,\ldots,x_n)\\
=\sum_{\alpha}(-1)^{h(\alpha)} (x_{\alpha_1},\ldots x_{\alpha_n},
y_1,\ldots, y_m).
\end{multline}
This last identity is derived from an equality that is trivial
unless  $n,m\geq 2$. Nevertheless, with the convention that $(x)=x$
the identity (\ref{eq:three}) is still valid and non-trivial for all
$n,m\geq 1$. Indeed, the case $n=1$ gives the antisymmetry relation
(\ref{eq:one}) and some of its consequences. Setting $m=1$, we
obtain the projection of formula~(\ref{eq:jacobi}). In particular,
$m=1$ and $n=2$ gives the Jacobi identity (\ref{eq:two}) after one
application of (\ref{eq:one}).

\begin{theorem}\label{thm:axioms}
Given a vector space $\ll$ with a set of multilinear operations
satisfying the identities (\ref{eq:three}) for all $n,m\geq 1$,
there exists a Lie algebra $\lie$ and a subalgebra $\subalg\subset \lie$ such
that $\lie=\subalg\oplus\ll$ and the operations on $\ll$ are projections of the
iterated Lie brackets in $\lie$ along $\subalg$.
\end{theorem}

\begin{proof}
Let $a_1,\ldots,a_n,\ldots$ be a basis for $\ll$, and $\lie$ --- the free
Lie algebra generated by the $a_i$.

Define the projection $\pi:\lie\to \ll$ by sending the iterated bracket $[a_{i_1},\ldots,a_{i_k}]$ to $(a_{i_1},\ldots,a_{i_k})$.
To show that this is well-defined we have to verify that the
identities in the free Lie algebra $\lie$ are projected to identities
in $\ll$.

Since $\lie$ is free, all the identities among the elements
$[a_{i_{1}},\ldots,a_{i_k}]$ are consequences of the antisymmetry
and the Jacobi identity in $\lie$. Using (\ref{eq:rewrite2}), the
antisymmetry identity can be re-written as
$$\sum_{\alpha}(-1)^{h(\alpha)}
[x_{\alpha_1},\ldots x_{\alpha_n}, y_1,\ldots,
y_m]=-\sum_{\beta}(-1)^{h(\beta)} [y_{\beta_1},\ldots y_{\beta_m},
x_1,\ldots,x_n].$$ The corresponding identity for the operations in
$\ll$ is obtained directly from $(\ref{eq:three})$ using the
antisymmetry of the binary bracket on $\ll$.

We see that $\pi$ is well-defined. Set $\subalg=\ker{\pi}$. Clearly, $\subalg$
is spanned by elements of the form
$$[a_{i_{1}},\ldots,a_{i_k}]-(a_{i_{1}},\ldots,a_{i_k}).$$
It follows from (\ref{eq:three}) that the Lie bracket of two elements
of this form is projected by $\pi$ into $0$.
\end{proof}

We shall call the Lie envelope constructed in the proof of Theorem~\ref{thm:axioms} the {\em free Lie envelope of $\ll$}. Let $\m$ be the maximal of all the ideals of $\lie$ that are contained in $\subalg$. Then 
$$\lie/\m = \subalg/\m\oplus \ll$$
is a Lie envelope of $\ll$. We shall call it the {\em standard Lie envelope of $\ll$}.

\begin{proof}[Proof of Theorem~\ref{nilenvelope}]
It is obvious that a flat Sabinin algebra with a nilpotent Lie envelope is itself nilpotent. 

For the converse statement, let $\lie$ be the free Lie envelope of $\ll$. If $\ll$ is nilpotent of class $n$, all the operations $(x_1, \ldots, x_k)$ for $k>n$ are zero. In particular, all the iterated Lie brackets in $\lie$ of weight greater than $n$ lie in $\subalg$.
Since these Lie brackets generate the ideal $\lie_{n+1}$, we see that $\lie/\lie_{n+1}$ is a nilpotent Lie envelope for $\ll$ with
$$\lie/\lie_{n+1}=\subalg/\lie_{n+1}\oplus \ll.$$
The nilpotency class of $\lie/\lie_{n+1}$ is the same as that of $\ll$. 
\end{proof}

The proof of Theorem~\ref{nilenvelope} also shows that the standard Lie envelope of a nilpotent Sabinin algebra $\ll$ is nilpotent of the same class as $\ll$, since the ideal $\lie_{n+1}$ is contained in the ideal $\m$.

\subsection{The standard Lie envelope}
The standard Lie envelope can be characterized as follows:
\begin{proposition}
Let $\lie=\subalg\oplus\ll$ be a Lie envelope of $\ll$ such that
\begin{itemize}
\item $\ll$ generates $\lie$ as a Lie algebra; 
\item no non-trivial ideal of $\lie$ is contained in $\subalg$.
\end{itemize}
Then $\lie$ is isomorphic to the standard Lie envelope.
\end{proposition}
\begin{proof}
Write $\lie_f=\subalg_f\oplus\ll$ for the free Lie envelope of $\ll$. If $\lie=\subalg\oplus\ll$ is a Lie envelope generated by $\ll$, there is a surjective homomorphism $\phi: \lie_f\to\lie$ which is identity on $\ll$. If $\m_f$ is the maximal of all the ideals of $\lie_f$ that are contained in $\subalg_f$, $\phi(\m_f)\subseteq\subalg$ is an ideal in $\lie$. In particular, $\phi$ vanishes on $\m_f$ and descends to a surjective homomorphism of the standard Lie envelope 
$$\phi':\lie_f/\m_f\to\lie.$$
Its kernel is contained in $\subalg_f/\m_f$ and, hence, must be trivial; therefore $\phi'$ is an isomorphism.
\end{proof}

For special classes of Sabinin algebras (Malcev algebras, Lie triple systems, Bol algebras) the standard Lie envelopes can be constructed using generalized derivations, see \cite{Sabinin_book}. 

\subsection{Multiplication Hopf algebras: the free case}
One general method of constructing Lie envelopes for Sabinin algebras is via the algebras of multiplications; we consider it in this and the following subsection.


Let $X$ be a set. We shall use the following notation:
\begin{itemize}
\item $W(X)$ -- the set of all non-empty words in $X$;
\item $\kk\{X\}$ -- the free unital non-associative algebra on $X$; 
\item $\Mlt = \Mlt\{X\}$ -- the associative algebra of self-maps of $\kk\{X\}$ generated by the left and right multiplications by all $z\in \kk\{X\}$.
\end{itemize}
In what follows, we shall write $L_z$ and $R_z$ for the operators of multiplication by $z$ on the left and on the right, respectively, in an arbitrary algebra. Thus, $\Mlt$ is the free associative algebra on the set $$\{L_w,R_w \mid w \in W(X)\}.$$
It is an associative Hopf algebra with 
\begin{eqnarray*}
\Delta(L_w) = \sum L_{w_{(1)}} \otimes L_{w_{(2)}} && \Delta(R_w) = \sum R_{w_{(1)}} \otimes R_{w_{(2)}}\\
\epsilon(L_w) = \epsilon(w) 1 && \epsilon(R_w) = \epsilon(w) 1 \\
S(L_w)\colon z \mapsto w\backslash z && S(R_w)\colon z \mapsto z/w
\end{eqnarray*}
for all $w \in W(X)\cup \{1\}$. Clearly,
\begin{displaymath}
\Delta(f(z)) = \sum f_{(1)}(z_{(1)}) \otimes f_{(2)}(z_{(2)}).
\end{displaymath}
As a Hopf algebra,  $\Mlt$ is cocommutative, pointed and irreducible.

Let  $\pi^L_+\colon \Mlt \rightarrow \Mlt$ be the linear map\footnote{A possibly confusing expression $f_{(1)}(1)$ means ``$f_{(1)}$ evaluated at $1$''.}
\begin{equation}\label{piplus}
\pi^L_+(f) = \sum  S(L_{f_{(1)}(1)}) f_{(2)},
\end{equation}
and define $\Mlt_+ = \Mlt\{X\}_+ = \pi^L_+(\Mlt)$.
\begin{proposition}
We have
\begin{enumerate}
\item $(\pi^L_+)^2 = \pi^L_+$;
\item  $\Mlt_+$ is the sum of all the subcoalgebras of $\Mlt$ contained in $\{ f \in \Mlt \mid f(1) = \epsilon(f) 1 \}$;
\item $\Mlt_+$ is a Hopf subalgebra of $\Mlt\{X\}$;
\item the map $L_{\kk\{X\}} \otimes \Mlt_+ \rightarrow \Mlt$ given by $L_z \otimes f \mapsto L_zf$ is a coalgebra isomorphism.
\end{enumerate}
\end{proposition}
\begin{proof}
Since $$\Delta \circ \pi^L_+ = \pi^L_+ \otimes \pi^L_+,$$ $\Mlt_+$ is a subcoalgebra of $\Mlt$. Any element $f$ in a subcoalgebra contained in  $$\{ f \in \Mlt \mid f(1) = \epsilon(f) 1 \}$$ satisfies 
$$\pi^+_L(f) =  \sum  S(L_{f_{(1)}(1)}) f_{(2)} = \epsilon(f_{(1)}) f_{(2)} = f,$$ which shows that all such subcoalgebras are contained in $\Mlt_+$. Since $\Mlt_+$ is itself a  subcoalgebra of this form, (2) is proved.

For the sake of convenience, write $M_z$ for either $L_z$ or $R_z$, and denote by $M'_z$ the operator $S(M_z)$. 
Given $w_1,\dots, w_n \in W(X)\cup\{1\}$ we have
$$
S\left(\pi^L_+(M_{w_1}\cdots M_{w_n})\right) = \sum M'_{w_{n(1)}} \cdots M'_{w_{1(1)}} L_{M_{w_{1(2)}}\cdots M_{w_{n(2)}}(1)}
$$
which shows that $S(\Mlt_+)$ is a subcoalgebra contained in $\{ f \in \Mlt \mid f(1) = \epsilon(f) 1 \}$ and, as a consequence, that $\Mlt_+$ is preserved by $S$. Since a product of subcoalgebras is again a subcoalgebra we conclude that $\Mlt_+$ is a Hopf subalgebra.

Any $f \in \Mlt$ can be written as $$f = \sum L_{f_{(1)}(1)}\pi^L_+(f_{(2)}).$$
The projection $\pi^L_+(L_{w}f)$ of $L_w f$ with $w \in W(X)$ and $f \in \Mlt_+$ vanishes since $$\pi^L_+(L_{w}f) = \sum S(L_{L_{w_{(1)}}f_{(1)}(1)})L_{w_{(2)}}f_{(2)} = \sum \epsilon(f_{(1)}) S(L_{w_{(1)}})L_{w_{(2)}}f_{(2)} = \epsilon(w)f = 0.$$ Therefore, given a linear combination  $f_0 + L_{w_1}f_1 + \cdots + L_{w_n}f_n = 0$ with distinct $w_1,\dots, w_n \in W(X)$  it is sufficient to apply  $\pi^L_+$ to get  $f_0 = 0$. Moreover, since in $\kk\{X\}$ the equality 
$w_1 a_1 +\ldots +w_n a_n =0$ with distinct $w_1,\dots, w_n \in W(X)$ implies that all $a_i=0$, we have $f_1 = \cdots = f_n = 0$. This proves (4).
\end{proof}

A similar statement can be proved for right multiplications; one only needs to replace  the map $\pi^L_+$ by  $\pi^R_+ \colon f \mapsto \sum S(R_{f_{(1)}(1)}) f_{(2)}$.

\medskip

Primitive elements of $\Mlt$ y $\Mlt_+$ form Lie algebras $\PMlt =\PMlt\{X\}$ and $\PMlt_+ = \PMlt\{X\}_+$, respectively. Clearly,
\begin{displaymath}
\PMlt = L_{\Prim\kk\{X\}} \oplus \PMlt_+ =  R_{\Prim\kk\{X\}} \oplus \PMlt_+.
\end{displaymath}

\medskip

Being a coalgebra, $\kk\{X\}$ carries the co-radical filtration $\kk\{X\} = \cup_{i\geq 0} \kk\{X\}_i$, where $\kk\{X\}_i$ is the subspace spanned by the products of at most $i$ primitive elements, which can have arbitrarly high degree in $X$. Since $\Prim \kk\{X\}$  is the free Sabinin algebra on $X$, we shall use the notation $\Sab\{X\}$ for it.  The vector space $L_{\Sab\{X\}}$  can be identified with $\Sab\{X\}$. Being a complement to a Lie subalgebra in a Lie algebra it carries the structure of a flat Sabinin algebra. We shall prove in the end of this section that this structure coincides with the Mikheev-Sabinin brackets on $\Sab\{X\}$.

\begin{proposition}
$\Mlt_+$ preserves the co-radical fitration on $\kk\{X\}$. 
\end{proposition}
\begin{proof}
It is sufficient to show that $\PMlt_+$ preserves the co-radical filtration. For $f \in \PMlt_+$ and $a \in \Sab\{X\}$ we have $$f(1) = \epsilon(f) 1 = 0$$ and $$\Delta(f(a)) = 
\sum f_{(1)}(a_{(1)}) \otimes f_{(2)}(a_{(2)}) = f(a) \otimes 1 + 1 \otimes f(a).$$ Assuming that $f$ preserves the co-radical filtration up to degree $k$ and using the identity
 $$\Delta(f(z)) = \sum f_{(1)}(z_{(1)}) \otimes f_{(2)}(z_{(2)}) = \sum f(z_{(1)}) \otimes z_{(2)} + z_{(1)}\otimes f(z_{(2)})$$ we see that $f$ preserves the co-radical filtration up to degree $k+1$.
\end{proof}
We stress that, since applying repeatedly $f_1,f_2,\dots \in \PMlt_+$ to $z \in \kk\{X\}_i$ produces  a linear combination of monomials which involve at most $i$ primitive elements, these elements must be of increasing degrees in $X$.

\medskip

Since each operator in $\Mlt_+$  is a linear combination of (possibly empty) compositions of operators in $\PMlt_+$, the following result shows that the Lie algebra $\PMlt_+$ encodes the Shestakov-Umirbaev operations in $\Sab\{X\}$:

\begin{proposition}
We have $\Sab\{X\} = \spann\langle f(x) \mid f \in \Mlt_+,\,  x \in X \rangle $. In fact,
\begin{displaymath}
\Sab\{X\}_{m+1} = \spann\langle f_1 \cdots f_j(x) \mid f_1,\dots, f_j \in \PMlt_+,\, j\geq m, \, x \in X\rangle.
\end{displaymath}
\end{proposition}

\begin{proof}
The spaces spanned by $$\left\{ 1, \sum S(L_{x_{(1)}y_{(1)}})L_{x_{(2)}}L_{y_{(2)}} \mid x,y \in \kk\{X\}\right\}$$ and $$\{1,  L_a - R_a \mid a \in \Sab\{X\}\}$$ are subcoalgebras whose elements  $f$ satisfy $f(1) = \epsilon(f) 1$; as a consequence, they are contained in $\Mlt_+$. The result now follows from the definition of the Shestakov-Umirbaev primitive operations.
\end{proof}

Now, let $\pi = \mathrm{Id} - \pi^L_+$ be the projection of $\PMlt$ to $L_{\Sab\{X\}}$. Define the Mikheev-Sabinin brackets on $L_{\Sab\{X\}}$ by the formulae (\ref{pi}) and (\ref{brackets}). 
\begin{proposition}\label{coincidence}
With the natural identification of $L_{\Sab\{X\}}$ with $\Sab\{X\}$, the Mikheev-Sabinin brackets on $L_{\Sab\{X\}}$ coincide with the Mikheev-Sabinin brackets on $\Sab\{X\}$. 
\end{proposition}
In particular, $\PMlt$  is a Lie envelope for $L_{\Sab\{X\}}$. Note that while $L_{\Sab\{X\}}$ is flat, $\Sab\{X\}$ is not.
\begin{proof}
First note that by (\ref{piplus}) $$\pi^L_+ (L_y L_z - L_z L_y) = - L_{[y,z]} + L_y L_z - L_z L_y$$
for $y,z\in\Sab\{X\}$ since the commutator of two primitive elements is primitive. Hence, from (\ref{brackets}) we have
$$\langle L_y, L_z \rangle = - (L_y, L_z) = -\pi (L_y L_z- L_z L_y) = -  (L_y L_z- L_z L_y) + \pi^L_+ (L_y L_z - L_z L_y) = - L_{[y,z]} = L_{\langle y, z \rangle}.$$
Now one can use the induction on $n$ to show that, in general,
$$\langle L_{x_1},\ldots L_{x_n}; L_y, L_z \rangle = L_{\langle {x_1},\ldots {x_n}; y, z \rangle}$$
for $x_i, y,z$ in $\Sab\{X\}$, the induction step being practically identical to the case $n=0$.
\end{proof}

\subsection{Multiplication Hopf algebras: the general case}\label{multgen}
Let now $\ll$ be a not necessarily free Sabinin algebra and $X$ a set of Sabinin algebra generators for $\ll$. The natural epimorphism of Hopf algebras  $$e: \kk\{X\} \rightarrow U(\ll)$$ can be extended to an epimorphism of unital algebras
$$e:\Mlt \rightarrow \Mlt(\ll)$$
 by 
 $$L_w \mapsto L_{e(w)}\quad\text{and}\quad R_w \mapsto R_{e(w)}$$
for all $w \in W(X)$; here $\Mlt(\ll)$ denotes the multiplication algebra of $U(\ll)$. Clearly, $e(f(z)) = e(f)(e(z))$ and, hence,
\begin{displaymath}
\ll_{m+1} = e(\Sab\{X\}_{m+1}) = \spann\langle e(f_1)\cdots e(f_j)(x) \mid f_1,\dots, f_j \in \PMlt_+, \, j\geq m,\, x\in\ll\rangle.
\end{displaymath} 

In the same way we define $\Mlt_+(\ll), \PMlt(\ll)$ and  $\PMlt_+(\ll)$ as the images under $e$ of $\Mlt$, $\PMlt$ and $\PMlt_+$ respectively. These algebras do not depend on the generating set $X$. It is clear that
\begin{displaymath}
\PMlt(\ll) = L_\ll \oplus \PMlt_+(\ll).
\end{displaymath}
The Lie algebra $\PMlt_+(\ll)$ preserves the co-radical filtration of $U(\ll)$ since this filtration is the image of the corresponding filtration on $\kk\{X\}$. It follows from Proposition~\ref{coincidence} that $\PMlt(\ll)$ is a Lie envelope for the Mikheev-Sabinin brackets of $\ll$.

\medskip

In the case of nilpotent Sabinin algebras there is a special kind of generating sets.
\begin{proposition}
A nilpotent Sabinin algebra  $\ll$ is generated by any basis in a complement of  $\ll_2$ in $\ll$.
\end{proposition}
\begin{proof}
Assume that $\ll_n \neq 0 = \ll_{n+1}$ and let $\B_1$ be such a base. By induction we have that $\ll/\ll_n$ is generated by  $\{a + \ll_n \mid a \in \B_1\}$ so that $\ll$ is generated by $\B_1$ and $\ll_n$. Since $\ll_n$ is central we see that $\ll_n\subseteq \ll_2$ is contained in the subalgebra generated by $\B_1$. 
\end{proof}

\begin{lemma}
For a nilpotent Sabinin algebra $\ll$ 
\begin{enumerate}
\item $\PMlt(\ll)$ is a nilpotent Lie algebra;
\item $\PMlt_+(\ll)$ consists of locally nilpotent maps.
\end{enumerate}
\end{lemma}
\begin{proof}
Let $\lie= \PMlt(\ll)$, $\lie_+ = \PMlt_+(\ll)$ and let   $\lie_{n+1}$ be the subspace generated by all the products of at least  $n+1$ elements in $\lie$.  Observe that  $f \in \lie$ and $u \neq 0$ imply $N(f(u)) > N(u)$. Therefore, if  $\ll_{n+1} = 0$ 
the ideal $\lie_{n+1}$ annihilates any $\lie$-submodule of $U(\ll)$ that contains $1$, as $\lie_{n+1}(1) \subseteq \ll_{n+1} = 0.$
Since $U(\ll)$ is generated as an $\lie$-module by $1$ we conclude that  $\lie_{n+1} = 0$.

Part (2) is the consequence of the fact that $\lie_+$ preserves the co-radical filtration, but increases the value of $N$ on nontivial elements. In particular, applying a sufficient number of elements of $\lie_+$ one can annihilate any element of  $U(\ll)$.
\end{proof}

\begin{theorem}\label{thm:Lie}
For any nilpotent Sabinin subalgebra $\ll$ of a finite-dimensional algebra $A$ the subspace $\ll_2$ generates a nilpotent subalgebra of $A$.
\end{theorem}
\begin{proof}

Assume that  $\ll_n \neq 0 = \ll_{n+1}$ with $n \geq 2$. Without loss of generality we can also assume that $A = U(\ll)/I$ for a certain ideal $I$ of finite codimension. Denote by $\lie(A), \lie_+(A)$ the Lie algebras defined in the same fashion as 
$\lie=\PMlt(\ll)$ and $\lie_+=\PMlt_+(\ll)$ but taking instead of $e\colon \kk\{X\} \rightarrow U(\ll)$ the composition of $e$ with the projection of $U(\ll)$ to $A$. Clearly, $\lie(A)$ and $\lie_+(A)$ are homomorphic images of  $\lie$ and $\lie_+$, respectively. Since $A$ is of finite dimension, $\lie_+(A)$ consists of nilpotent transformations of $A$.

Let us prove by induction on  the dimension of $\ll$ that there exists a natural number $m$ such that $\ll^m_2 \subseteq I$. Choose a non-zero element $c \in \ll_n$; it lies in the center of $A$. As $n\geq 2$, we have  
$$c = f_1(a_1)+\cdots + f_r(a_r)$$
with $a_1,\dots, a_r \in \ll$ and $f_1,\dots, f_r \in \l_+(A)$. In particular, $c = f(1)$ with $$f=[f_1,L_{a_1}] + \cdots + [f_r,L_{a_r}] \in [\lie(A),\lie(A)].$$ As $\lie(A)$ is a nilpotent Lie algebra, $f$ is nilpotent. We can write $f= f_+ + L_{c}$ where $f_+ \in \lie(A)_+$. Since $L_{c}$ lies in the centre of $\lie(A)$ we have $[f,f_+] = 0$ and, hence,  as $f$ and $f_+$ are nilpotent $L_{c}$ also is.   This proves that $c^m \in I\subset U(\ll)$ for $m$ sufficiently large.

Choose a basis  $\B = \B_1 \sqcup \cdots \sqcup \B_n$ as in Section~\ref{here-we-define-B}. By the induction assumption applied to $\ll/\kk c$ we see that there exists  $m$  such that $\ll^{m}_2 \subseteq I +  U(\ll)c$. Let $$U_i =U(\ll)_i = \kk 1 + \ll^1 + \cdots + \ll^i.$$ There exists a natural number $d_1$ such that $\ll^{m}_2 \subseteq I + U_{d_1}c$. If we show that for any natural number $d_i$ there exist $d_{i+1}$ and $e_{i+1}$ such that  $$((U_{d_i}\ll_2)\cdots)\ll_2 \subseteq I + U_{d_{i+1}}c$$ ($e_{i+1}$ copies of $\ll_2$) then, as $c$ is central, the product $((\ll_2 \ll_2)\cdots)\ll_2$ ($d = m+e_2+\cdots +e_m$ factors) is contained in $I + U_{d_m}c^m \subseteq I$. Any element of $\ll^{nd}_2$ can be expressed in terms of the Poincar\'e-Birkhoff-Witt basis as a linear combination of monomials with at least $d$ factors, all of them in $\ll_2$, and, hence, $\ll^{nd}_2 \subseteq I$. Taking  $m=nd$ one would prove that $\ll^m_2 \subseteq I$.

Consider the basis $\B' = \B$ with the opposite order. Each element $w$ of the basis of $U(\ll)$ associated with $\B$ that has  $d_i$ elements in $\B_1$ and at least  $e_{i+1} = n(d_i + m)$ in $\B_2 \sqcup \cdots \sqcup \B_n$ can be expressed as a linear combination with non-zero coefficients of the elements $w'_1,\dots, w'_l$ of the basis of $U(\ll)$ associated with $\B'$. The number of factors in $w'_j$ is at least $d_i+m$, of which at most  $d_i$ belong to $\B_1$. Therefore, at least $m$ factors lie in  $\B_2 \sqcup \cdots \sqcup \B_n$. Given the choice of the order on $\B'$ we have that  $w'_j \in I + U(\ll)c$. As a consequence,  $((U_{d_i}\ll_2)\cdots)\ll_2 \subseteq I + U(\ll)c$, where $\ll_2$ appears $e_{i+1}$ times. Since $((U_{d_i}\ll_2)\cdots)\ll_2$ is finite-dimensional, we can choose an adecuate $d_{i+1}$.
\end{proof}

\begin{remark} Guided by the analogy with Lie algebras, one may conjecture that there is a version of Theorem~ \ref{thm:Lie} for {\em solvable} Sabinin algebras. However, we do not yet have an adecuate definition of a solvable Sabinin algebra. The following example illustrates that a naive approach to the notion of solvability does not produce desired results.

Consider the unital algebra $A$ spanned by $1,a,b$ with the products given by 
\begin{displaymath}
aa = 0,\quad ab = b, \quad ba = 0 \mbox{ y } bb = b.
\end{displaymath}
The subspace $\kk b$ is an ideal of $A$ and we have
\begin{displaymath}
(b,a,b) = (ba)b-b(ab) = -b.
\end{displaymath}
The subspace $\ll = \kk a + \kk b$ is a Sabinin subalgebra of $A$; however, the subalgebra generated by $\ll_2 = \kk b$ is not nilpotent. Nevertheless, any Sabinin algebra operation vanishes on $\ll_2$. Also, the Lie algebra $\PMlt(\ll)$ is solvable since the multiplication algebra of $A$ has this property.
\end{remark}

\section{Further comments}
\subsection{The Ado theorem for nilpotent loops}
The Ado theorem for Lie algebras is a statement about Lie algebra representations. There are several ways to think of representations of Sabinin algebras; thus, in principle, the Ado theorem may generalize (or fail to do so) in more than one way.

Recall that the Ado theorem for groups claims that any finite-dimensional Lie group can be locally embedded into the group of units of a finite-dimensional algebra. In contrast to the version for the Lie algebras, there is no ambiguity about the non-associative generalization of this statement: an analytic loop satisfies the Ado theorem if it can be locally embedded into the local loop of invertible elements in a non-associative algebra. (Note that invertible elements of a non-associative algebra do not always form a globally defined loop, a counterexample being the Cayley-Dickson algebra $\mathbb{A}_4$.) In the case of nilpotent loops the following global version of the Ado theorem holds:

\begin{theorem}\label{thm:ado2} Let $L$ be a torsion-free nilpotent loop of class $n$. Then the composition
$$L \hookrightarrow \kk [L] \to \kk [L]/ \overline{\kk [L]}\,^{n+1}$$
is injective. Here $ \overline{\kk [L]}$ is the augmentation ideal of $\kk [L] $.
\end{theorem}
The algebra $\kk [L]/ \overline{\kk [L]}\,^{n+1}$ is finite-dimensional whenever the abelianization of $L$ is of finite rank.
Its invertible elements  form a loop, not just a local loop. 
\begin{proof}
If $L$ is torsion-free nilpotent loop of class $n$, the Jennings theorem \cite{M2} implies that the dimension subloop $D_{n+1}L$ is trivial. In other words, the only element $g\in L$ such that $g-1\in \overline{\kk [L]}\,^{n+1}$ is $g=1$. If $g_1-g_2\in \overline{\kk [L]}\,^{n+1}$, we have that  $1-g_1\backslash g_2 \in \overline{\kk [L]}\,^{n+1}$ too, and, hence, $g_1=g_2$.
\end{proof}

We should point out that the proof of Theorem~\ref{thm:ado} is the linearized version of the argument which establishes the non-associative generalization of the Jennings theorem \cite{M2}. 

\begin{lemma}\label{lem:csc-tf}A  connected and simply-connected nilpotent loop is torsion-free. 
\end{lemma}
\begin{proof}
Assume $x\in L$ is a non-trivial torsion element and $L$ is connected, simply-connected and nilpotent, of smallest dimension possible. Let $S$ be a uniparametric subgroup in the centre of $L$; it is also torsion-free and simply-connected as it integrates a subalgebra of the Sabinin algebra of $L$. Since $x\notin S$, its image in $L/S$ is a non-trivial torsion element. However, $L/S$ is nilpotent and simply-connected and, hence, must be torsion-free. 
\end{proof}

\begin{theorem}
A connected analytic loop $L$ is nilpotent and simply-connected if and only if there exists a finite-dimensional algebra $A = \kk 1 \oplus I$, with $I$ a nilpotent ideal, and an injective homomorphism of analytic loops
$$\varphi \colon L \rightarrow 1 + I,$$ 
whose differential  at the unit is also injective.
\end{theorem}

\begin{proof}
If $L$ is nilpotent and simply-connected, it is torsion-free by Lemma~\ref{lem:csc-tf}. The existence of $A$ and $\varphi$ follows then from Theorem~\ref{thm:ado2}.

The converse can be proved in the same manner as for groups. First, observe that $1+ I$ is an analytic loop whose Sabinin algebra coincides with the Shestakov-Umirbaev Sabinin algebra  $\SU(I)$ of $I$ (see \cite{MP3}) and such that $\exp \colon I \rightarrow 1+I$ is a diffeomorphism. Let $\ll$ be the Sabinin algebra of $L$. The loop homomorphism  $\varphi$ induces an injective Sabinin algebra homomorphism $\varphi'\colon \ll \rightarrow \SU(I)$. Write $B(\ll)$ for the simply-connected loop with Sabinin algebra $\ll$; as $L$ is connected, there is a surjective homomorphism $\pi\colon B(\ll) \rightarrow L$ whose differential at the unit is the isomorphism $\pi'$ of the respective Sabinin algebras. Since the diagram
\begin{displaymath}
\begin{CD}
B(\ll) @>\pi>> L @>\varphi>> 1 + I \\
@A{\exp}AA @A\exp AA @A\exp AA \\
\ll @>\pi'>> \ll @>\varphi'>> I
\end{CD}
\end{displaymath}
commutes, we have $\varphi(L) = \exp(\varphi'(\ll))$ so that  $L$ is homeomorphic to $\ll$ and, hence, is simply-connected. The loop  $1+I$ is nilpotent and $\varphi$ is injective; as a consequence,  $L$ is nilpotent too.
\end{proof}

\subsection{Complete Hopf algebras} The equivalence between the categories of formal Lie groups and Lie algebras can be realized in two steps. One associates to a formal Lie group the Hopf algebra of distributions on it; then, the  primitive elements in this Hopf algebra form a Lie algebra. In the case of nilpotent groups, the central role of Hopf algebras is even more transparent, see \cite[Appendix]{Quillen}. In brief, they are the conceptual tool behind the Baker-Campbell-Hausdorff formula. 

The theory of Hopf algebras remains the same in the non-associative context, see \cite{MP3, MPSh, PI2}. In particular, one may define complete Hopf algebras and retrace the steps of \cite{Quillen}. In fact, we have done part of this task here and the rest presents no challenge. One important point which we should mention briefly is the fact that complete Hopf algebras can be used to define the Malcev completion of a loop. Let $\widehat{\Q}[L]$ be the completion of the rational loop algebra of a loop $L$ with respect to the augmentation ideal. The Malcev completion of $L$ can be defined as the loop of the group-like elements in $\widehat{\Q}[L]$. The details are completely similar to the case of groups.

\end{document}